%% file: stochastic_lya_ACC2017_arxiv.tex
\documentclass[journal]{IEEEtran}
\usepackage{cite}

\ifCLASSINFOpdf
\else
\fi

\usepackage{graphicx} 
\usepackage{epsfig} 
\usepackage{times} 
\usepackage{amsmath} 
\usepackage{amssymb} 
\usepackage{multicol}
\usepackage{url}
\newtheorem{theorem}{Theorem}

\newtheorem{assumption}[theorem]{Assumption}
\usepackage{color}
\usepackage{mathrsfs}
\usepackage{multirow,minipage}
\usepackage{caption}
\usepackage{subcaption}
\usepackage{wrapfig}

\newtheorem{definition}[theorem]{Definition}

\newtheorem{lemma}[theorem]{Lemma}

{ 
\newtheorem{remark}[theorem]{Remark}

}
\newenvironment{proof}[1][Proof]{\textbf{#1.} }{\ \hspace*{\fill} \rule{0.5em}{0.5em}}

\def\oP{P^{1}}

\def\R{\mathbb{R}}

\newcommand{\sub}[2]{(#1)_{(#2)}}

\newcommand{\bi}{\begin{itemize}}
\newcommand{\ei}{\end{itemize}}
\newcommand{\bd}{\begin{displaymath}}
\newcommand{\ed}{\end{displaymath}}
\newcommand{\be}{\begin{eqnarray*}}
\newcommand{\ee}{\end{eqnarray*}}

\begin{document}
%
\title{Transfer Operator-Based Approach for Optimal Stabilization of Stochastic System}

%
%
%


\author{ Apurba Kumar Das, Arvind Raghunathan, and  Umesh Vaidya
\thanks{Financial support from the National Science Foundation grant CNS-1329915 and ECCS-1150405 is gratefully acknowledged. U. Vaidya is with the Department of Electrical \& Computer Engineering,
Iowa State University, Ames, IA 50011.}}
\maketitle
\begin{abstract}
In this paper we develop linear transfer Perron-Frobenius operator-based approach for optimal stabilization of stochastic nonlinear systems. One of the main highlights of the proposed transfer operator based approach is that both the theory and computational framework developed for the optimal stabilization of deterministic dynamical systems in \cite{raghunathan2014optimal} carries over to the stochastic case with little change. 
The optimal stabilization problem is formulated as an infinite dimensional linear program. Set oriented numerical methods are proposed for the finite dimensional approximation of the transfer operator and the controller.  Simulation results are presented to verify the developed framework. 
\end{abstract}
\section{Introduction}
Transfer operator-based methods have attracted lot of attention lately for problems involving dynamical systems analysis and design. In particular, transfer operator-based methods are used for identifying steady state dynamics of the system from the invariant measure of transfer operator, identifying almost invariant sets, and coherent structures \cite{Dellnitz00,Dellnitiz_almostinvariant,froyland2009almost}. The spectral analysis of transfer operators are also applied for reduced order modeling of dynamical systems with applications to building systems, power grid, and fluid mechanics \cite{budivsic2012applied,amit_koopmanestiamtor}. Operator-theoretic methods have also been successfully applied to  address design problems in control dynamical systems. In particular, transfer operator methods are used for almost everywhere stability verification, control design, nonlinear estimation, and for solving optimal sensor placement problem \cite{VaidyaMehtaTAC,raghunathan2014optimal,Rajeev_continuous_time_journal,Vaidya_CLM,vaidya2007observability,sinha2016operator,MVCDC05}.

In this paper, we continue with the long series of work on the application of transfer operator methods for stability verification and stabilization of nonlinear systems. We develop an analytical and computational framework for the application of transfer operator methods for the stabilization of stochastic nonlinear systems. In \cite{vaidya2015stochastic}, we introduced Lyapunov measure for stability verification of stochastic nonlinear systems. We proved that the existence of the Lyapunov measure verifies weaker set-theoretic notion of almost everywhere stochastic stability for discrete-time stochastic systems. Weaker notion of almost everywhere stability was introduced in \cite{Rantzer01} for continuous time deterministic systems and in \cite{van2006almost} for continuous time stochastic systems.  In this paper we extend the application of Lyapunov measure for optimal stabilization of stochastic nonlinear systems. Optimal stabilization of stochastic systems is posed as an infinite dimensional linear program. 
Set-oriented numerical methods are used for the finite dimensional approximation of the transfer operator and the linear program. A key advantage of the proposed transfer operator-based 
approach for stochastic stability analysis and controller synthesis is that all the stability results 
along with the computation framework carries over from the deterministic systems  \cite{VaidyaMehtaTAC,raghunathan2014optimal} to the stochastic systems. The only 
difference in the stochastic setting is that the transfer Perron-Frobenius operator is defined 
for the stochastic system.

The results developed in this paper draw parallels from following papers.    
Lasserre, Hern\'{a}ndez-Lerma, and co-workers \cite{LassHernBook,approxInfLP}  formulated the control of Markov processes as a solution of the HJB equation.  In \cite{Hernandez_ocp, Lasserre_ocp,Gait_ocp}, 
solutions to stochastic
and deterministic optimal control problems are proposed, using
a linear programming approach or using a sequence of LMI relaxations. Our paper also draws some connection to research on optimization and
stabilization of controlled Markov chains discussed in \cite{Meyn_sadhana}. Computational techniques based on the viscosity solution of the HJB equation is proposed for the approximation of the 
value function and optimal controls in \cite[Chapter VI]{viscosity_solnHJB_book}.

Our proposed method, in particular the computational approach, draws some similarity with the above discussed references on the approximation of the solution of the HJB equation \cite{cell-cell1, Junge_Osinga, viscosity_solnHJB_book,Junge_scl_05}.
Our method, too, relies on discretization of state space to obtain globally optimal stabilizing  control. However, our proposed approach differs from the above references in the following two fundamental ways. The first main difference arises due to adoption of non-classical weaker set-theoretic notion of almost everywhere stability for optimal stabilization. 
The second main difference compared to references \cite{Meyn_sadhana} and \cite{viscosity_solnHJB_book} is in the use of the discount factor $\gamma>1$ in the cost function.   
The discount factor plays an important role in controlling the effect of finite dimensional discretization or the approximation process on the true solution. In particular, by allowing for the discount factor, $\gamma$, to be greater than one, it is possible to ensure that the control obtained using the finite dimensional approximation is {\it truly} stabilizing 
for the nonlinear system \cite{Vaidya_CLM,arvind_ocp_online}. 

The paper is organized as follows. In section \ref{section_lymeas} we present brief overview of results from \cite{vaidya2015stochastic} on Lyapunov measure for stochastic stabilization. In \ref{section_lymeasoptimal} results on application of Lyapunov measure for optimal stabilization are presented. In section \ref{section_computation}, computational framework based on set-oriented numerical methods for finite dimensional approximation of Lyapunov measure and optimal control is presented. Simulation results are presented in section \ref{section_examples} followed by conclusions in section \ref{section_conclusion}.

\section{Lyapunov measure for stochastic stability analysis}\label{section_lymeas}

Consider the discrete-time stochastic  system,
\begin{eqnarray}
x_{n+1}=T(x_n,\xi_n),\label{rds}
\end{eqnarray}
where $x_n\in X\subset \mathbb{R}^d$ is a  compact set.  The random vectors, $\xi_0,\xi_1,\ldots$, are assumed  independent identically distributed (i.i.d) and  takes values in $W$ with the following probability distribution,
\begin{equation}
{\rm Prob}(\xi_n\in B)=v(B),\;\;\forall n, \;\;B\subset W, \label{def:probMeasure}
\end{equation}
and is the same for all $n$ and $v$ is the probability measure.
The system mapping $T(x,\xi)$ is assumed  continuous in $x$ and for every fixed $x\in X$, it is measurable in $\xi$. The initial condition, $x_0$, and the sequence of random vectors, $\xi_0,\xi_1,\ldots$, are assumed  independent. 
The basic object of study in our proposed approach to stochastic stability is a linear transfer, the Perron-Frobenius operator,  defined as follows:
\begin{definition}[Perron-Frobenius (P-F) operator] \label{def-PF}Let ${\cal M}(X)$ be the space of finite measures on $X$.
The Perron-Frobenius operator, $\mathbb{P}: {\cal M}(X)\to {\cal M}(X)$, for stochastic dynamical system (\ref{rds}) is given by 
\begin{eqnarray}
[\mathbb{P}_T\mu](A)=\int_X \left\{\int_W \chi_A(T(x,y))dv(y)\right\}d\mu(x)
\label{pfequation}
\end{eqnarray}
 for $\mu\in {\cal M}(X)$, and $A\in {\cal B}(X)$, where ${\cal B}(X)$ is the Borel 
 $\sigma$-algebra on $X$, $T^{-1}_{\xi}(A) = T^{-1}(A,\xi)$ is the inverse image of the set $A$, 
 and $\chi_A(x)$ is an indicator function of set $A$.  
\end{definition}

\begin{assumption}\label{assume_equilibrium} We assume  $x=0$ is an equilibrium point of  system (\ref{rds}), i.e.,
$T(0,\xi_n)=0,\;\;\;\forall n,$
for any given sequence of random vectors $\{\xi_n\}$.
\end{assumption}

\begin{assumption}[Local Stability]\label{assume_local} We assume  the trivial solution, $x=0$, is locally stochastic, asymptotically stable. In particular, we assume there exists a neighborhood ${\cal O}$ of $x=0$, such that for all $x_0\in \cal O$,
\[{\rm Prob}\{T^n(x_0,\xi_0^n)\in {\cal O}\}=1,\;\;\forall n\geq 0,\]
and
\[{\rm Prob}\{\lim_{n\to \infty} T^n(x_0,\xi_0^n)=0\}=1.\]
where $\xi_0^n$ notation is used to define the sequence of random variable $\{\xi_0,\ldots,\xi_n\}$.
\end{assumption}
Assumption \ref{assume_equilibrium} is used in the decomposition of the P-F operator in section (\ref{section_decompose}) and  Assumption \ref{assume_local} is used in the proof of Theorem \ref{theorem_converse}. We will use the notation $U(\epsilon)$ to denote the $\epsilon$ neighborhood of the origin for any positive value of $\epsilon>0$. We have $0\in U(\epsilon)\subset {\cal O}$.

We introduce the following definitions for stability of the stochastic dynamical system (\ref{rds}).

\begin{definition}[a.e. stochastic stablity with geometric decay]\label{def_aeas_geometric} For any given $\epsilon>0$,  let $U(\epsilon)$ be the $\epsilon$ neighborhood of the equilibrium point, $x=0$. The equilibrium point, $x=0$, is said to be almost everywhere, almost sure stable with geometric decay with respect to finite measure,  $m\in {\cal M}(X)$,    if there exists $0<\alpha(\epsilon)<1$, $0<\beta<1$, and $K(\epsilon)<\infty$, such that
\[m \{x\in X: Prob \{T^n(x,\xi_0^n)\in B\}\geq \alpha^n\}\leq K \beta^n, \]
for all sets $B\in {\cal B}(X\setminus U(\epsilon))$, such that $m(B)>0$.
\end{definition}

We introduce the following definition of absolutely continuous 
measures.
\begin{definition}  [Absolutely continuous measure] A measure $\mu$
is absolutely continuous with respect to another measure, $\vartheta$
denoted as $\mu\prec\vartheta$, if $\mu(B) = 0$ for all $B\in{\cal B}(X)$ with $\vartheta(B) =
0$.
\end{definition}

\subsection{Decomposition of the P-F operator}\label{section_decompose}
Let $E=\{0\}$. Hence, $E^c=X\setminus E$. We write $T: E\cup E^c\times W\to X$. For any set $B\in {\cal B}(E^c)$, we write
{\small
\begin{eqnarray}[\mathbb{P}_T\mu](B)&=&\int_X \int_W \chi_B(T(x,y))dv(y)d\mu(x)\nonumber\\&=&\int_{E^c} \int_W \chi_B(T(x,y))dv(y)d\mu(x).\end{eqnarray}}
This is because  $T(x,\xi)\in B$ implies  $x\notin E$.
Since  set $E$ is invariant, we define the restriction of the P-F operator on the complement set $E^c$. Thus, we can define the restriction of the P-F operator on the measure space ${\cal M}(E^c)$ as follows:

\begin{equation}
[\mathbb{P}_T^1\mu](B)=\int_{E^c} \int_W \chi_B(T(x,y))dv(y)d\mu(x),\label{pfrest1}
\end{equation}
for any set $B\in{\cal B}(E^c)$ and $\mu\in {\cal M}(E^c)$.

 Next, the restriction $T:E\times W\rightarrow E$ can also be used
to define a P-F operator denoted by
\begin{equation}
[\mathbb{P}_T^0\mu](B)=\int_{B} \chi_B(T(x,y)) dv(y) d\mu(x),\label{pfrest2}
\end{equation}
where $\mu\in{\cal M}(E)$ and $B\subset {\cal B}(E)$.

The above considerations suggest a representation of the P-F
operator, $\mathbb{P}$, in terms of $\mathbb{P}_0$ and $\mathbb{P}_1$.
Indeed, this is the case, if one considers a splitting of the measured
space,
${\cal M}(X)={\cal M}_0\oplus{\cal M}_1$,
where ${\cal M}_0 := {\cal M}(E)$, ${\cal M}_1 := {\cal
M}(E^c)$, and $\oplus$ stands for the direct sum.

The splitting defined in
above equation implies that the P-F operator has a lower-triangular
matrix representation given by
\begin{equation}
\mathbb{P}_T=\left[ \begin{array}{cc} \mathbb{P}_T^0 & 0 \\
\times & \mathbb{P}_T^1 \end{array} \right]. \label{eq:splitP}
\end{equation}


Following  definition of Lyapunov measure is introduced for a.e. stability verification of system (\ref{rds}).
\begin{definition}[Lyapunov measure]\label{definition_Lyameas} A Lyapunov measure, $\bar \mu\in{\cal M}(X\setminus U(\epsilon))$, is defined as any positive measure  finite outside the $\epsilon$ neighborhood of equilibrium point and satisfies
\begin{eqnarray}
[\mathbb{P}_T^1 \bar \mu](B)< \gamma \bar \mu(B) \label{lya_meas}
\end{eqnarray}
for $0<\gamma\leq 1$ and for all sets $B\in {\cal B}(X\setminus U(\epsilon))$.
\end{definition}

The following theorem provides the
condition for a.e. stochastic stability with geometric decay.
\begin{theorem} \label{theorem_converse} An attractor set $\cal A$
for the system (\ref{rds}) is a.e. stochastic stable with
geometric decay (Definition \ref{def_aeas_geometric}) with respect to finite measure $m$, if and only if for all $\epsilon>0$ there exists
a non-negative measure $\bar \mu$ which is finite on
${\cal B}(X\setminus U(\epsilon))$  and satisfies
\begin{equation}
\gamma[ {\mathbb P}_T^1\bar \mu](B)-\bar \mu(B)=-m(B)\label{LME}
\end{equation} for all  sets
$B\subset X\setminus U(\epsilon)$ and for some $\gamma>1$.
\end{theorem}
\begin{proof}
We omit the proof here due to space constraints. However, the proof follows exactly along the lines of the proof for deterministic systems \cite{raghunathan2014optimal}
\end{proof}

\section{Lyapunov measure for optimal stabilization}\label{section_lymeasoptimal}

We consider the stabilization of stochastic dynamical system
\[x_{n+1}=T(x_n,u_n,\xi_n)=:T_{\xi_n}(x_n,u_n)\]
where $x_n\in X\subset \mathbb{R}^q$ is the state, $u_n \in U\subset
\mathbb{R}^d$ is the control input, and $\xi_n\in W\subset \mathbb{R}^p$ is a random variable. The sequence of random variables $\xi_0,\xi_1,\ldots$ are assumed to independent identically distributed (i.i.d.) as in~\eqref{def:probMeasure}.
For each fixed value of $\xi$  the mapping $T_{\xi}:X\times U\to X$ is assumed to be continuous in $x$ and $u$, and for every fixed values of $x$ and $u$ it is measurable in $\xi$. Both $X$ and $U$ are assumed  compact. The
objective is to design a deterministic feedback controller, $u_n=K(x_n)$, to optimally 
stabilize the attractor set $\cal A$. 

We define
the feedback control mapping $C: X\rightarrow Y:= X\times U$ as
$C(x)=(x,K(x))$.
We denote by ${\cal B}(Y)$ the Borel-$\sigma$ algebra on $Y$ and ${\cal M}(Y)$ the vector space of real valued measures on ${\cal B}(Y)$. For any $\mu \in {\cal M}(X)$, the control mapping $C$  can be used to define a measure, $\theta\in {\cal M}(Y)$, as follows:
\begin{eqnarray}
&\theta(D):=[\mathbb P_C \mu](D)=\mu (C^{-1}(D))\nonumber\\
&[\mathbb P_{C^{-1}} \theta](B):=\mu(B)=\theta (C(B))\label{ess},
\end{eqnarray}
for all sets $D\in {\cal B}(Y)$ and $B\in {\cal B}(X)$. Since $C$ is an injective function with $\theta$ satisfying (\ref{ess}),  it follows from the theorem on disintegration of measure \cite{disintegration} (Theorem 5.8)  there exists a unique disintegration $\theta_x$ of the measure $\theta$  for $\mu$ almost all $x\in X$, such that $
\int_Y f(y)d\theta(y)=\int_X\int_{C(x)} f(y)d\theta_x(y)d\mu(x)$,
for any Borel-measurable function $f:Y\to \mathbb R$. In particular, for $f(y)=\chi_D(y)$, the indicator function for the set $D$, we obtain
$\theta(D)=\int_X \int_{C(x)}\chi_D(y) d\theta_x(y) d\mu(x)=[\mathbb{P}_C\mu](D).$
Using the definition of the feedback controller
mapping $C$, we write 
\[x_{n+1}=T(x_n, K(x_n),\xi_n)=T( C (x_n),\xi_n)=:T_{\xi_n}\circ C(x_n).\]
The system mapping $T: Y\times W \rightarrow X$ can be associated with P-F operators ${\mathbb
P}_{T}: {\cal M}(Y)\rightarrow {\cal M}(X)$ as \[
[{\mathbb P}_{T} \theta](B)= \int_{Y} \left\{\int_W  \chi_B(T(y,\xi)) d v(\xi) \right\}d\theta(y).\]

For the feedback control system $T_\xi \circ C: X\times W \rightarrow X$, the P-F operator 
can be written as a product of ${\mathbb P}_{T_\xi}$ and ${\mathbb P}_C$. In particular, we obtain 
\[
[{\mathbb P}_{T_\xi \circ C} \mu](B)=\int_Y \left\{\int_W\chi_B(T(y,\xi)) dv(\xi) \right\}d [{\mathbb P}_C \mu](y) 
\]
\[
[{\mathbb P}_{T} {\mathbb P}_C \mu](B)=\int_X \int_{C(x)}\left\{\int_W \chi_{B}(T(y,\xi))dv(\xi)\right\}d\theta_x(y)d\mu(x).
\]

The P-F operators, $\mathbb{P}_T$ and $\mathbb{P_C}$, are used to define their restriction, $\mathbb{P}_T^1:{\cal M}({\cal A}^c\times U)\to {\cal M}({\cal A}^c)$, and $\mathbb{P}_C^1:{\cal M}({\cal A}^c)\to {\cal M}({\cal A}^c\times U)$ to the complement of the attractor set, respectively,  similar to  Eqs. (\ref{pfrest1})-(\ref{pfrest2}).

\begin{assumption}
We assume  there exists a feedback controller mapping $C_0(x)=(x,K_0(x))$, which locally stabilizes the
invariant set ${\cal A}$, i.e., there exists a neighborhood $V$ of $\cal A$
such that $T\circ C_0(V)\subset V$ and $x_n\rightarrow \cal A$ for all
$x_0\in V$; moreover ${\cal A} \subset U(\epsilon)\subset V$.
\end{assumption}

Our objective is to construct the optimal stabilizing controller
for almost every initial condition starting from $X_1$. Let
$C_{1}: X_1\rightarrow Y$ be the stabilizing control map for $X_1 (:=X\setminus U(\epsilon))$.
The control mapping $C: X\to X\times U$ can be written as follows:
\begin{equation}C(x)=\left \{\begin{array}{ccl}C_0(x)=(x,K_0(x))&
{\rm for}&x\in U(\epsilon)\\
C_1(x)=(x,K_1(x))&{\rm for}&x\in X_1.
\end{array}\right.
\end{equation}

Furthermore, we assume the feedback control system $T_\xi \circ C:
X\rightarrow X$ is non-singular with respect to the Lebesgue measure,
$m$ for fixed value of $\xi$. We seek to design the controller mapping, $C(x)=(x,K(x))$, such that the attractor set $\cal A$ is a.e. stable with geometric decay rate $\beta<1$, while minimizing the  cost function,
\begin{equation} {\cal
C}_C(B)=\int_{B}\sum_{n=0}^{\infty} \gamma^n E_{\xi_0^n}[G(C(x_n),\xi_n)]
dm(x),\;\;\;\;\label{cost_sum}
\end{equation}
where $x_0=x$, the cost function $G: X\times U\times  W \to {\mathbb R}$ is assumed a continuous non-negative real-valued function for each fixed value of $\xi$ and is assumed to be measurable w.r.t. $\xi$ for fixed values of $x$ and $u$. Furthermore, $G({\cal A},0, \xi)=0$ for all $\xi$, $x_{n+1}=T_\xi \circ C(x_n)$, and $0<\gamma<\frac{1}{\beta}$. The expectation $E_{\xi_0^n}$ in (\ref{cost_sum}) denotes expectation over the sequence of random variable $\{\xi_0,\xi_1,\ldots, \xi_n\}$.
  Note, that in the cost function (\ref{cost_sum}), $\gamma$ is allowed  to be greater than one and this is one of the main departures from the conventional optimal control problem, where $\gamma\leq 1$. Under the assumption that the controller mapping $C$ renders the attractor set a.e. stable with a geometric decay rate, $\beta<\frac{1}{\gamma}$, the cost function (\ref{cost_sum}) is finite.

 \begin{remark}
We will use the notion of the scalar product between continuous function $h\in {\cal C}^0(X)$ and measure $\mu\in {\cal M}(X)$ as $\left<h,d\mu\right>_X :=\int_X h(x)d\mu(x)$ \cite{Lasota}.
 \end{remark}

Let the controller mapping, $C(x)=(x,K(x))$, be such that the attractor set $\cal A$ for the feedback control system $T_\xi \circ C:X\to X$ is a.e. stable with geometric decay rate $\beta<1$. Then, the cost function (\ref{cost_sum}) is well defined for $\gamma<\frac{1}{\beta}$ and, furthermore, the cost of stabilization of the attractor set $\cal A$ with respect to Lebesgue almost every initial condition starting from the set $B\in {\cal B}(X_1)$ can be expressed as follows:
\begin{eqnarray}
&{\cal C}_C(B)= \int_{B}\sum_{n=0}^{\infty}\gamma^n E_{\xi_0^n}[G( C(x_n),\xi_n)]
dm(x)\nonumber\\&=\int_{{\cal Z}}G(y,\xi)d[\mathbb{P}_C^1 \bar \mu_B](y)dv(\xi)=
\left<G,d{\mathbb P}_C^1 \bar \mu_B dv \right>_{{\cal Z}},\label{inequality}
\end{eqnarray}
where, ${\cal Z}={\cal A}^c\times U\times W$, $x_0=x$ and $\bar \mu_B$ is the solution of the following  control
Lyapunov measure equation,
\begin{eqnarray}
\gamma [\mathbb{P}_{T}^1\cdot \mathbb{P}_C^1\bar \mu_B](D)-\bar
\mu_B(D)=-m_B(D), \label{control_Ly}
\end{eqnarray}
for all $D\in {\cal B}(X_1)$ and  where $m_B(\cdot):=m(B\cap \cdot)$ is a finite measure
supported on the set $B\in {\cal B}(X_1)$.


The minimum cost of stabilization is defined as the minimum over
all a.e. stabilizing controller mappings, $C$, with a geometric decay as follows:
\begin{equation}
{\cal C}^{*}(B)=\min_{C}{\cal C}_C(B).
\end{equation}
Next, we write the infinite dimensional linear program for the optimal stabilization of the attractor
set $\cal A$. Towards this goal, we first define the projection map,
$P_1: {\cal A}^c\times U\rightarrow {\cal A}^c$
as:
$P_1(x,u)=x,$
and denote the P-F operator corresponding to $P_1$ as
$\mathbb{P}_{P_1}:  {\cal M}({\cal A}^c\times U)\rightarrow
{\cal M}({\cal A}^c)$, which can be written as
$[{\mathbb P}^1_{P_1} \theta](D)=\int_{{\cal A}^c\times U}\chi_D(P_1(y))d\theta(y)=\int_{D\times U}d\theta(y)=\mu(D)$.
 Using this definition of projection mapping, $P_1$, and the
corresponding P-F operator, we can write the linear program for the optimal stabilization of set $B$
with  unknown variable
$\theta$ as follows:
\begin{eqnarray}
\min\limits_{\theta\geq 0} && \left<G, d\theta dv\right>_{{\cal A}^c \times U\times W}\nonumber\\
\mbox{s.t. } && \gamma [{\mathbb P}^1_T \theta](D)-[{\mathbb P}^1_{P_1}
\theta](D)=-m_B(D)\label{linear_program},
\end{eqnarray}
for $D\in {\cal B}(X_1)$.
\begin{remark}\label{remark_discount}
Observe  the geometric decay parameter satisfies $\gamma > 1$.
This is in contrast to most optimization problems studied in
the context of Markov-controlled processes, such as in Lasserre and
Hern\'{a}ndez-Lerma \cite{LassHernBook}. Average cost and
discounted cost optimality problems are considered in
\cite{LassHernBook, viscosity_solnHJB_book}. The additional flexibility provided by $\gamma>1$ guarantees the controller obtained from the finite dimensional approximation of the infinite dimensional program (\ref{linear_program}) also stabilizes the attractor set for control dynamical system. 
\end{remark}


\section{Computational approach}\label{section_computation}
We discretize the state-space and
control space  for the purposes of computations as described below.
Borrowing the notation from \cite{Vaidya_CLM}, let
${\cal X}_N := \{D_1,...,D_i,...,D_N\}$ denote a finite
partition of the state-space $X \subset \R^q$.
The measure space associated with ${\cal X}_N$ is $\R^N$.
We assume without loss of generality that the
attractor set, ${\cal A}$, is contained in $D_{N}$, that is,
${\cal A} \subseteq D_N$. The control space, $U$, is quantized and the control input
is assumed to take only finitely many control values from the
quantized set,
${\cal U}_M = \{u^1,\hdots,u^a,\hdots,u^M\}$,
where $u^a \in \R^d$.  The partition, ${\cal U}_M$, is identified with the vector
space, $\R^{d \times M}$. Similarly, the space of uncertainty, $W$, and the probability measure $v$ is quantized and are assumed to take  only finitely many values ${\cal W}=\{\xi^1,\ldots, \xi^\ell,\ldots, \xi^L\}$, and $\vartheta=\{v^1,\ldots,v^\ell,\ldots, v^L\}$ where $\xi^\ell\in \mathbb{R}^p$ and $0\leq v^\ell\leq 1$ for all $\ell$ and $\sum_{\ell=1}^L v^\ell=1$. The discrete probability measure on the finite dimensional uncertainty space is assigned as follows:
\[Prob\{\xi_n=\xi^\ell\}=v^\ell,\;\;\;\forall n,\;\;\;\ell=1,\ldots,L.\]
The space of uncertainty is identified with finite dimensional space $\R^{p\times L}$. 
 The system map that results from choosing the controls $u=u^a$ and uncertainty value $\xi=\xi^\ell$ is denoted by
$T_{u^a,\xi^\ell}$ and the corresponding P-F operator is denoted as
$P_{T_{u^a,\xi^\ell}} \in \R^{N \times N}$. Note that for system mapping $T_{u^a, \xi^\ell}$, the control on all sets of
the partition is $u(D_i) = u^a$, for all $D_i \in
{\cal X}_N$. For brevity of notation, we will denote the P-F matrix $P_{T_{u^a,\xi^\ell}} $ by $P_{T_{a,\ell}} $ and its entries are calculated as

\[
\sub{P_{T_{a,\ell}}}{ij} := \frac{m({T^{-1}_{u^a,\xi^\ell}}(D_j)\cap D_i)}{m(D_i)},
\] where $m$ is the Lebesgue measure and $\sub{P_{T_{a,\ell}}}{ij}$ denotes the
$(i,j)$-th entry of the matrix.
Since $T_{u^a, \xi^\ell}:X \rightarrow X$, we have  $P_{T_{a,\ell}}$ is a Markov matrix.
Additionally, $\oP_{T_{a,\ell}} : \R^{N-1} \rightarrow \R^{N-1}$
will denote the finite dimensional counterpart of the
P-F operator restricted to ${\cal X}_N \setminus D_N$,
the complement of the attractor set.
It is easily seen that $\oP_{T_{a,\ell}}$ consists of the first
$(N-1)$ rows and columns of $P_{T_{a,\ell}}$.


With the above quantization of the control space and partition of the
state space, the determination of the control $u(x) \in U$ (or
equivalently $K(x)$) for all $x \in {\cal A}^c$  has now been cast as a problem
of choosing $u_N(D_i) \in {\cal U}_M$ for all sets $D_i
\subset {\cal X}_N$.  The finite dimensional approximation of the
optimal stabilization problem \eqref{linear_program} is equivalent to
solving the following finite-dimensional LP:
\begin{eqnarray}
 \min\limits_{\theta^a,\mu \geq 0} &&
 \sum_{a=1}^M\left[ \sum_{\ell=1}^L v^\ell (G^{a,\ell})^{'} \right]\theta^a\nonumber\\
 \mbox{s.t.} && \gamma  \sum_{a=1}^{M} \left[\sum_{\ell=1}^L v^\ell (P_{T_{a,\ell}})^{'}\right]\theta^a
 - \sum_{a=1}^M\theta^a = -m
\label{lyaMeasLP1},
\end{eqnarray}
where we have used the notation $(\cdot)^{'}$ for the transpose operation,
$m \in \R^{N-1}$ and $\sub{m}{j} > 0$ denote the support of Lebesgue measure,
$m$, on the set $D_j$, $G^{a,\ell} \in \R^{N-1}$
is the cost defined on ${\cal X}_N\setminus D_N$ with
$\sub{G^{a,\ell}}{j}$ the cost associated with using control action $u^a$
on set $D_j$ with uncertainty value $\xi=\xi^\ell$; $\theta^a \in \R^{N-1}$ are, respectively, the
discrete counter-parts of infinite-dimensional measure quantities in
\eqref{linear_program}. We define following quantities 
\[G^a:=\sum_{\ell=1}^L v^\ell G^{a,\ell},\;\;\;P_{T_a}:=\sum_{\ell=1}^L v^\ell P_{T_{a,\ell}}\]
to rewrite finite-dimensional LP (\ref{lyaMeasLP1}) as follows:
\begin{equation}
 \min\limits_{\theta^a,\mu \geq 0} \mbox{ }
 \sum_{a=1}^M (G^{a})^{'} \theta^a,\;\;\;
 \mbox{s.t. } \gamma  \sum_{a=1}^{M} (P_{T_{a}})^{'}\theta^a
 - \sum_{a=1}^M\theta^a = -m
\label{lyaMeasLP},
\end{equation}

In the LP \eqref{lyaMeasLP}, we have not enforced the constraint,
\begin{equation}
\sub{\theta^a}{j} > 0 \mbox{ for exactly one } a \in \{1,...,M\},
\label{detControl}
\end{equation}
for each $j = 1,...,(N-1)$.  The above constraint ensures  the
control on each set in unique.  We prove in the following  the
uniqueness can be ensured without  enforcing the constraint,
provided the LP \eqref{lyaMeasLP} has a solution.
To this end, we introduce the dual LP
associated with the LP in \eqref{lyaMeasLP}.
The dual to the LP in \eqref{lyaMeasLP} is,
\vspace{-0.1in}
\begin{equation}
\max\limits_{V} \mbox{ } m^{'}V ,\;\;
\mbox{s.t. } V \leq \gamma\oP_{T_a}V + G^a \mbox{ } \forall
a = 1,...,M.
\label{lyaLP}
\end{equation}
In the above LP \eqref{lyaLP}, $V$ is the dual variable to the
equality constraints in \eqref{lyaMeasLP}.
\vspace{-0.1in}
\subsection{Existence of solutions to the finite LP}

We make the following assumption throughout this section.
\begin{assumption} \label{assumption_A}
There exists $\theta^{a} \in \R^{N-1} \;\forall\; a = 1,\ldots,M$,
such that the LP in \eqref{lyaMeasLP} is feasible for
some $\gamma > 1$.
\end{assumption}

\begin{lemma}\label{strongDuality}
Consider a partition ${\cal X}_N = \{D_1,\ldots,D_N\}$ of the state-space
$X$ with attractor set ${\cal A} \subseteq D_N$ and a quantization
 ${\cal U}_M = \{u^1,\ldots,u^M\}$ of the control space $U$.  Suppose
 Assumption \ref{assumption_A} holds for some $\gamma > 1$ and for $m, G > 0$.
Then, there exists an
optimal solution, $\theta$, to the LP \eqref{lyaMeasLP} and
an optimal solution, $V$, to the dual LP \eqref{lyaLP}
with equal objective values, ($\sum\limits_{a=1}^M(G^a)^{'}\theta^a
= m'V$) and $\theta, V$ bounded.
\end{lemma}
\begin{proof}
Refer to proof of Lemma 12 in~\cite{raghunathan2014optimal}.
\end{proof}


The next result shows the LP \eqref{lyaMeasLP} always admits an optimal solution
satisfying \eqref{detControl}.

\begin{lemma} \label{compLemma1}
Given a partition ${\cal X}_N = \{D_1,\ldots,D_N\}$ of the state-space,
$X$, with attractor set, ${\cal A} \subseteq D_N$, and a quantization,
${\cal U}_M = \{u^1,\ldots,u^M\}$, of the control space, $U$.  Suppose
Assumption \ref{assumption_A} holds for some $\gamma > 1$ and for $m, G > 0$.  Then, there
exists a solution $\theta \in \R^{N-1}$ solving \eqref{lyaMeasLP}
and $V \in \R^{N-1}$ solving \eqref{lyaLP} for any
$\gamma \in [1,\overline{\gamma}_N)$.  Further, the following hold at
the solution: 1) For each $j = 1,...,(N-1)$, there exists at least one $a_j \in 1,...,M$,
  such that $\sub{V}{j} = \gamma\sub{\oP_{T_{a_j}}V}{j} + \sub{G^{a_j}}{j}$
  and $\sub{\theta^{a_j}}{j} >0$. 2) There exists a $\tilde{\theta}$ that solves \eqref{lyaMeasLP},  such that for each $j = 1,...,(N-1)$, there is exactly one
$a_j \in 1,...,M$, such that $\sub{\tilde{\theta}^{a_j}}{j} > 0$ and
$\sub{\tilde{\theta}^{a'}}{j} = 0$ for $a' \neq a_j$.
\end{lemma}
\begin{proof}
Refer to proof  of Lemma 14 in~\cite{raghunathan2014optimal}.
\end{proof}

The following theorem states the main result.

\begin{theorem}\label{thmComp}
Consider a partition ${\cal X}_N = \{D_1,\ldots,D_N\}$ of the state-space,
$X$, with attractor set, ${\cal A} \subseteq D_N$, and a quantization,
${\cal U}_M = \{u^1,\ldots,u^M\}$, of the control space, $U$.
Suppose Assumption \ref{assumption_A} holds for some $\gamma > 1$ and for
$m, G > 0$. Then, the following statements hold: 1) there exists a bounded $\theta$, a solution to
  \eqref{lyaMeasLP} and a bounded $V$, a solution
  to \eqref{lyaLP}; 2) the optimal control for each set, $j= 1,...,(N-1)$, is given by
$ u(D_j) = u^{a(j)},\;\;\mbox{where } a(j) := \min\{a | \sub{\theta^a}{j} > 0\}$; 3) $\mu$ satisfying
$\gamma(\oP_{T_u})^{'}\mu - \mu = -m\;\;, \mbox{ where }
\sub{\oP_{T_u}}{ji} = \sub{\oP_{T_{a(j)}}}{ji}$
is the Lyapunov measure for the controlled system.
\end{theorem}
\begin{proof}
Refer to proof of Lemma 15 in~\cite{raghunathan2014optimal}.
\end{proof}

\section{Example: Inverted Pendulum on cart}\label{section_examples}

\input{sec_simulation_arxiv}

\section{Conclusions}\label{section_conclusion}
Transfer Perron-Frobenius operator-based framework is introduced for optimal stabilization of stochastic nonlinear systems. Weaker set-theoretic notion of almost everywhere stability is used for the design of optimal stabilizing feedback controller. The optimal stabilization problem is formulated as an infinite dimensional linear program. The finite dimensional approximation of the linear program and the associated optimal feedback controller is obtained using set-oriented numerics.

\bibliographystyle{IEEEtran}
\bibliography{ref}

%
%
%





\end{document}

%% file: sec_simulation_arxiv.tex
\begin{eqnarray}
\ddot{x} &=&  \cfrac{a\sin{(x)}-0.5m_{r}\dot x^2\sin{(2x)}-b\cos{(x)}u}{1.33-m_r\cos^2{(x)}} \nonumber\\
&-&2\zeta\sqrt{a}\dot x 
\label{system_dynamics}
\end{eqnarray}
where $g = 9.8, l = 0.5, m = 2, M = 8, \zeta = 0, m_r = \cfrac{m}{(m+M)}, a = \cfrac{g}{l}, b = \cfrac{m_r}{ml}$. The cost function is assumed to be $G(x,u)=x^2+\dot x^2+u^2$.   
 For uncontrolled system, $u = 0$, 
 \begin{wrapfigure}[12]{l}{4.1 cm}
\begin{center}
{{\includegraphics[height=1.2 in]{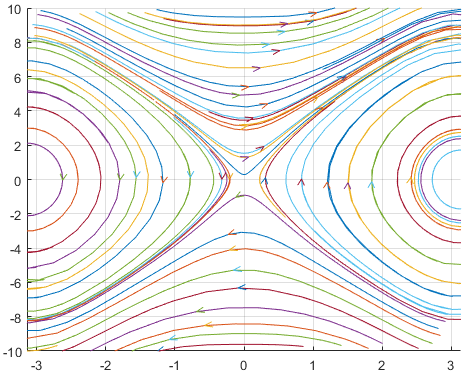}}}
\caption{Phase portrait of inverted pendulum on a cart}
\label{inv_pend_phsprt}
\end{center}
\end{wrapfigure}
 there are two equilibrium points, one equilibrium point at $(\pi,0)$ is stable in Lyapunov sense with eigenvalues of linearization on the $j\omega$ axis, the second equilibrium point at the origin is a saddle and unstable. In Fig. \ref{inv_pend_phsprt} we show the phase portrait for the uncontrolled system. The objective is to optimally stabilize the saddle equilibrium point at the origin. For the purpose of discretization we use $\delta t = 0.1$ as time discretization for the simulations. The state space $X$ is chosen to be limited in $[-\pi,\pi]\times [-10,10]$ and is partitioned into $70\times70=4900$ boxes. For constructing the P-F matrix $10$ initial conditions are located in each box. The control set is discretized as follows ${\cal U}=\{-80,-70\ldots,-10,0 ,10,\ldots, 70,80\}$. 
\begin{itemize}
 \item {\bf Case 1}: 
The damping parameter $\zeta$ is assumed to be random and uniformly distributed with mean zero and uniformly supported on the interval $[-\sigma,\sigma]$. 
Similarly, the range of random parameter $[-\sigma,\sigma]$ is divided into 10 uniformly spaced discrete values for random parameter $\xi$.
\item {\bf Case 2}: The parameter $b$ multiplying the control input is assumed to be Bernoulli random variable with statistics $Prob(b=1)=p$ and $Prob(b=0)=1-p$ for every time. 
\end{itemize}
In Fig. \ref{case1_fig1}, we show the plot for the Lyapunov measure, optimal control for $\sigma=0.1$.  In Fig. \ref{case1_fig2}, we show the plot for the optimal cost and the percentage of initial condition that are attracted to the origin. 
It is interesting to notice that the optimal cost along the stable manifold of the uncontrolled system is small whereas along the unstable manifold is large. This is because of the fact the optimal control is design to exploit the natural dynamics of the system since there is non-zero cost on the control efforts. The simulation result in Fig.\ref{case1_fig2}b are obtained by performing time domain simulation with eight initial conditions in each box iterated over $100$ time step. We notice that close to $100$ percentage of initial conditions are attracted to the origin.  In Fig. \ref{case1_fig3} we show the comparison of sample trajectory for the open loop and closed loop system. The sample trajectory shows that feedback controller is able to stabilize the origin. 

{\bf Case 2}: For case 2, we consider Bernoulli uncertainty for the input channel. In this case the random parameter $b$  can take only two values at every time instant i.e., $b=0$ or $b=1$. For this case we only show the plots for the percentage of initial conditions that can be optimally stabilized to the origin for two different values of erasure probabilities $1-p=0.15$ and $1-p=0.5$. Again the simulation results for this case are obtained by  performing time domain simulation with eight initial conditions in each box iterated over $100$ time step. From Fig. \ref{case2_fig1}, we notice that with erasure probability of $1-p=0.15$ more than $97\%$ of initial conditions are attracted to the origin. While for $1-p=0.5$ only  $66\%$  of points are attracted to the origin thereby indicating that the origin is not stabilized with erasure probability of $0.5$.

\begin{figure}
\begin{center}
\includegraphics[scale = 0.23]{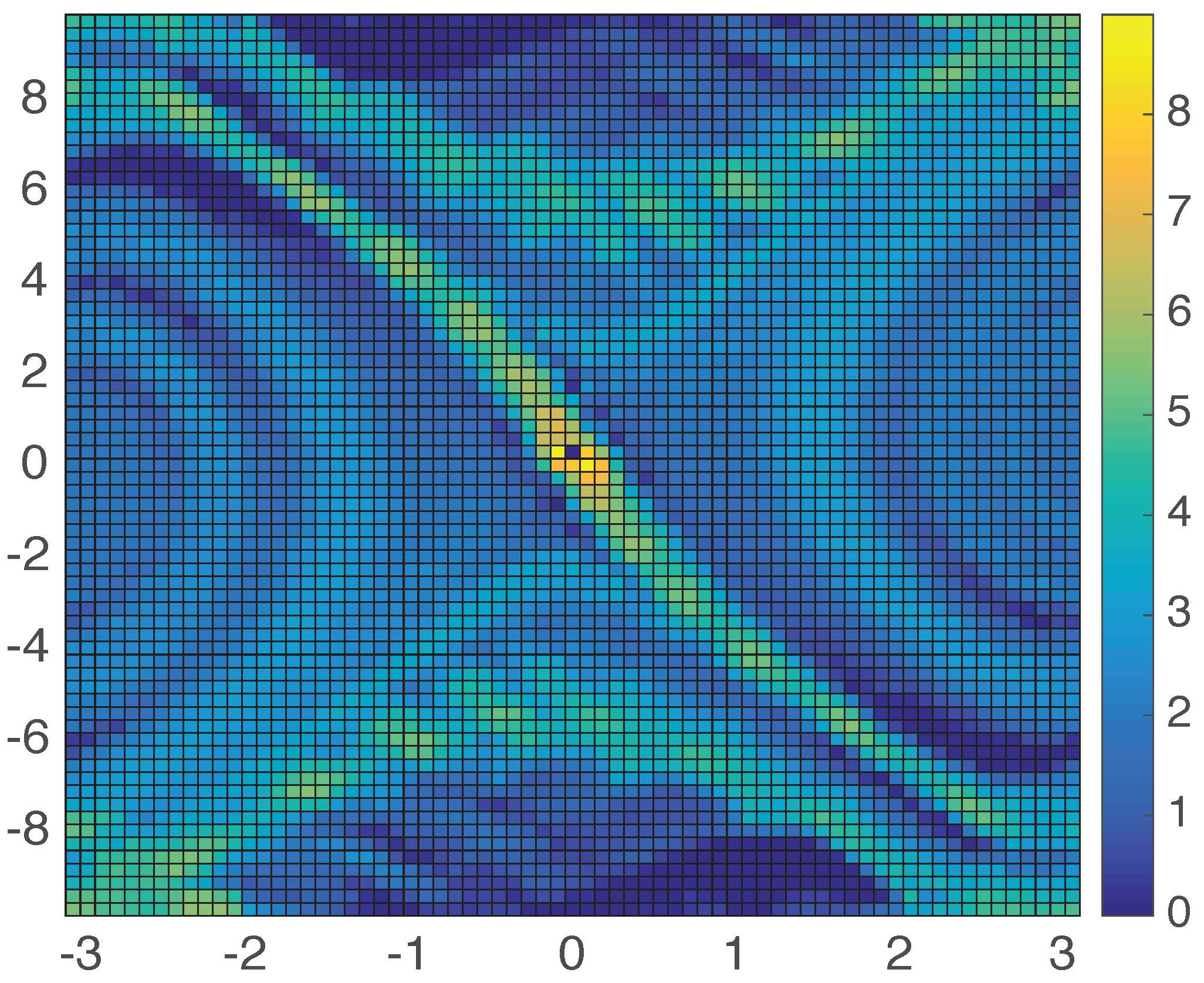}
\includegraphics[scale = 0.23]{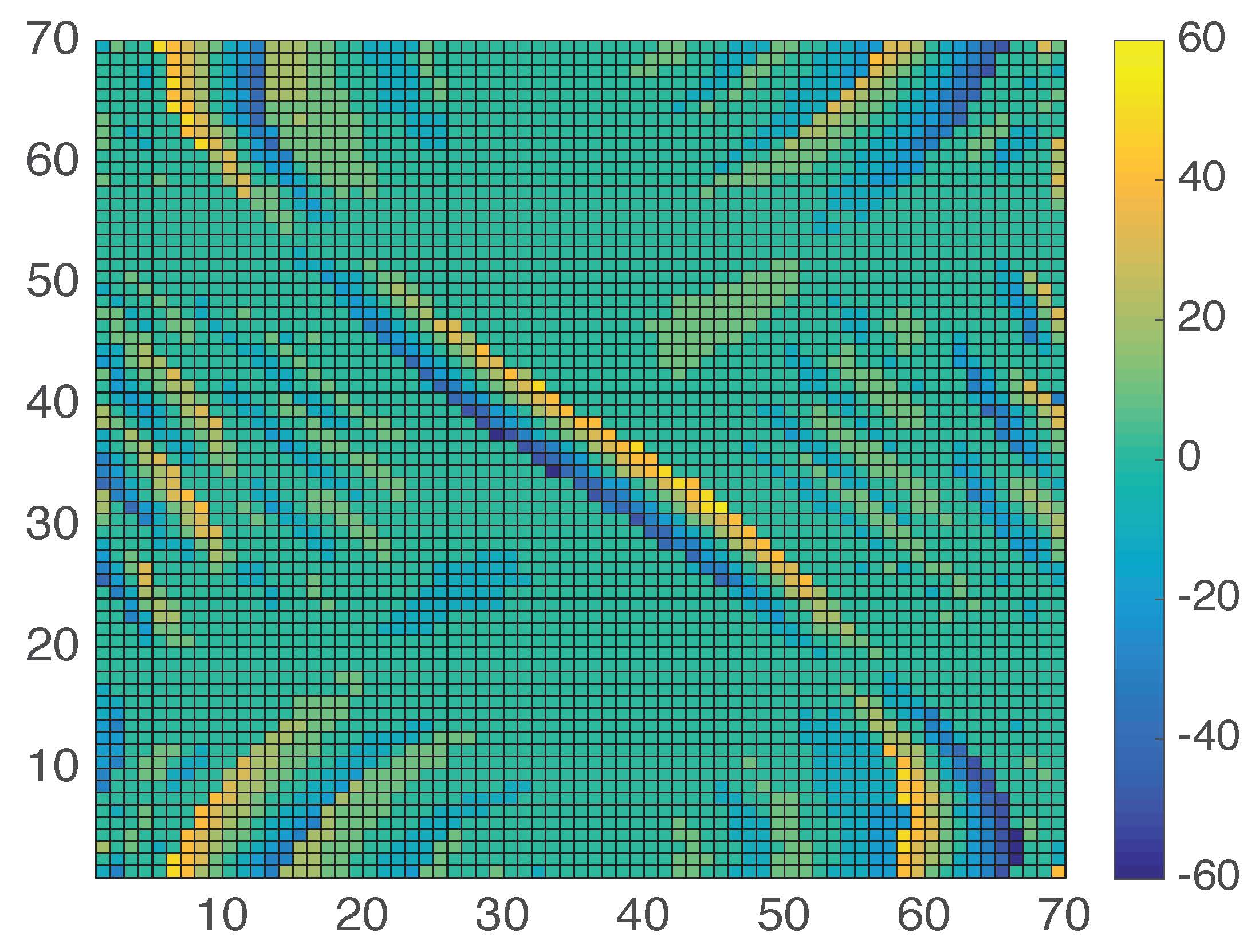}
\caption{Case 1: a) Lyapunov measure plot; b) Control values plot}
\label{case1_fig1}
\end{center}
\end{figure}

\begin{figure}
\begin{center}
\includegraphics[scale = 0.2]{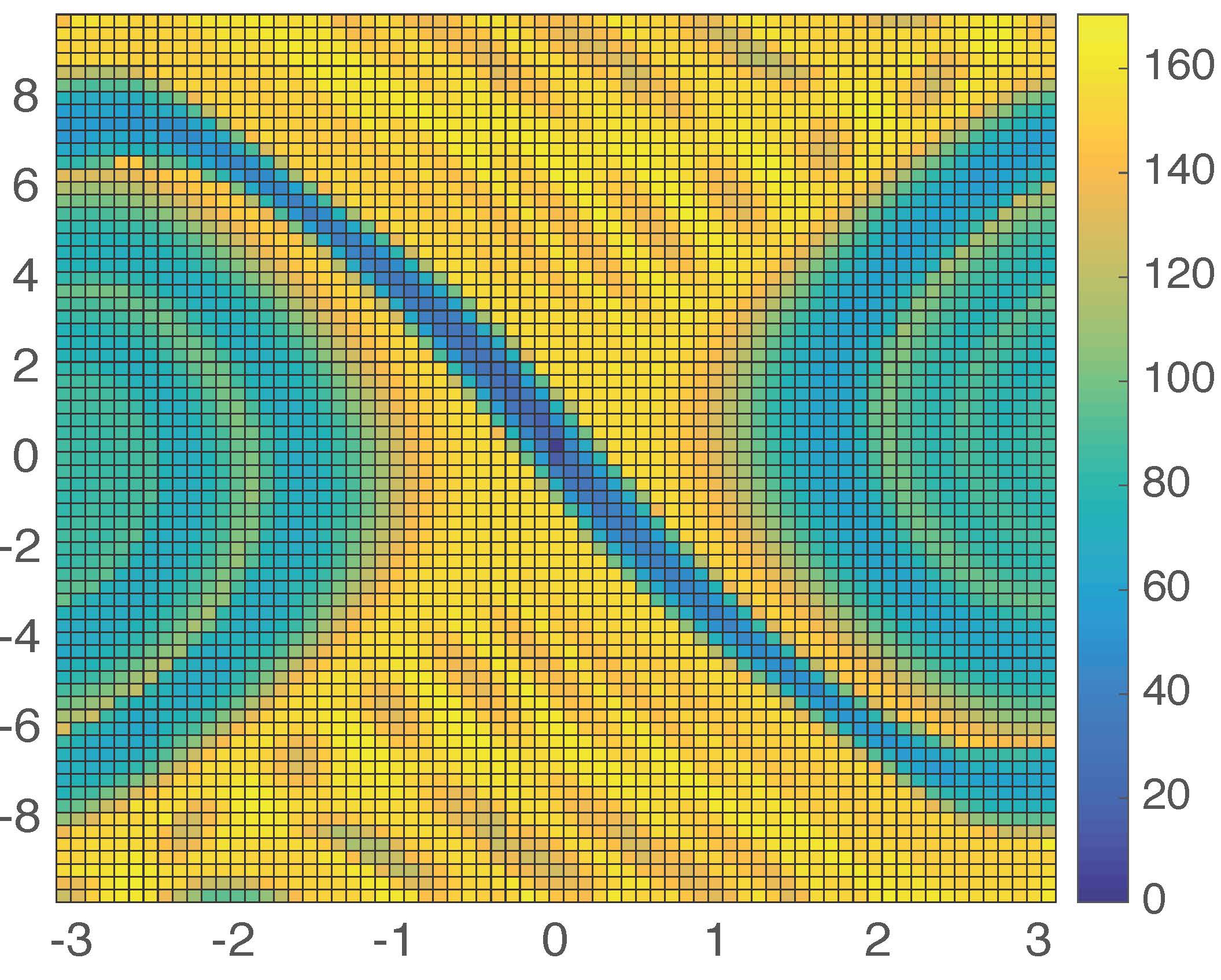}
\includegraphics[scale = 0.1]{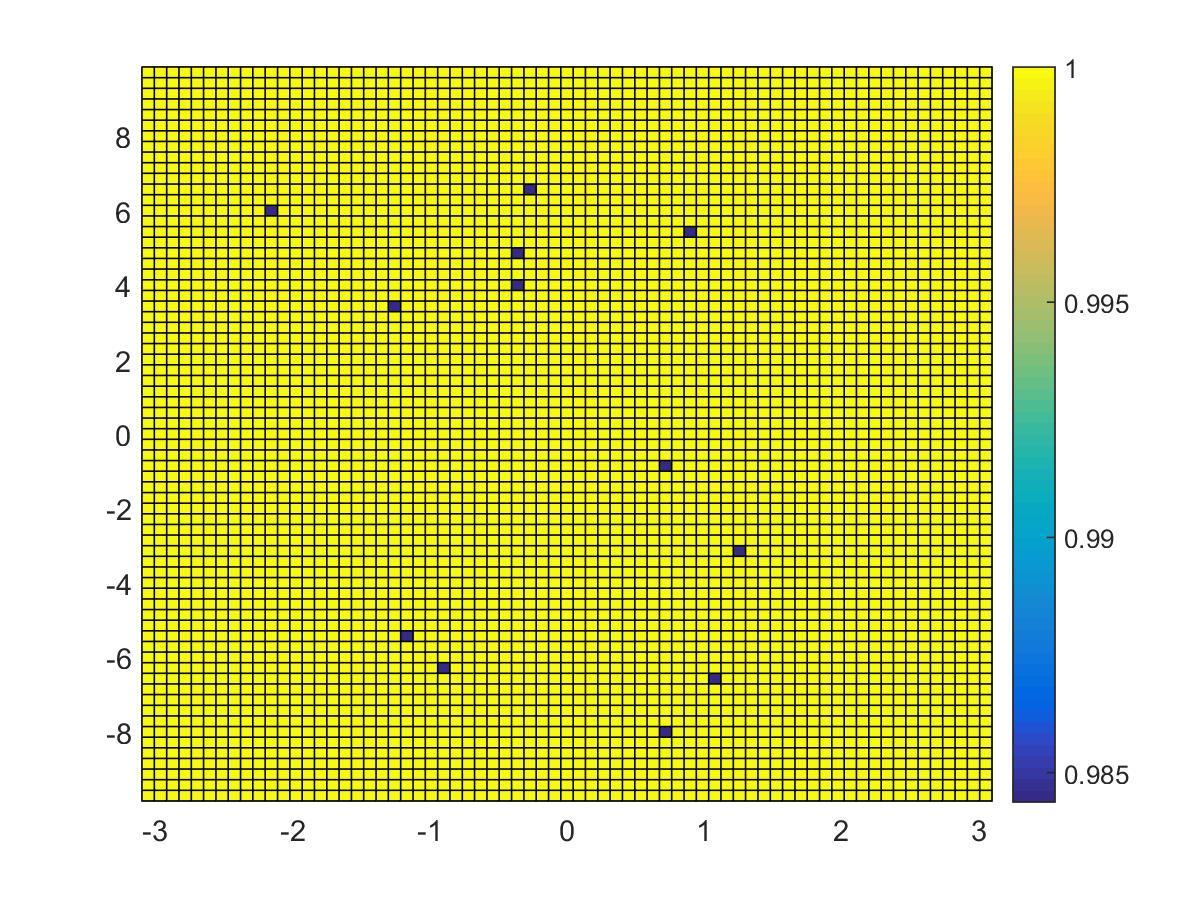}
\caption{Case 1: a) Optimal cost plot; b) Plot for percentage of initial condition optimally stabilized to origin}
\label{case1_fig2}
\end{center}
\end{figure}

\begin{figure}
\begin{center}
\includegraphics[scale = 0.23]{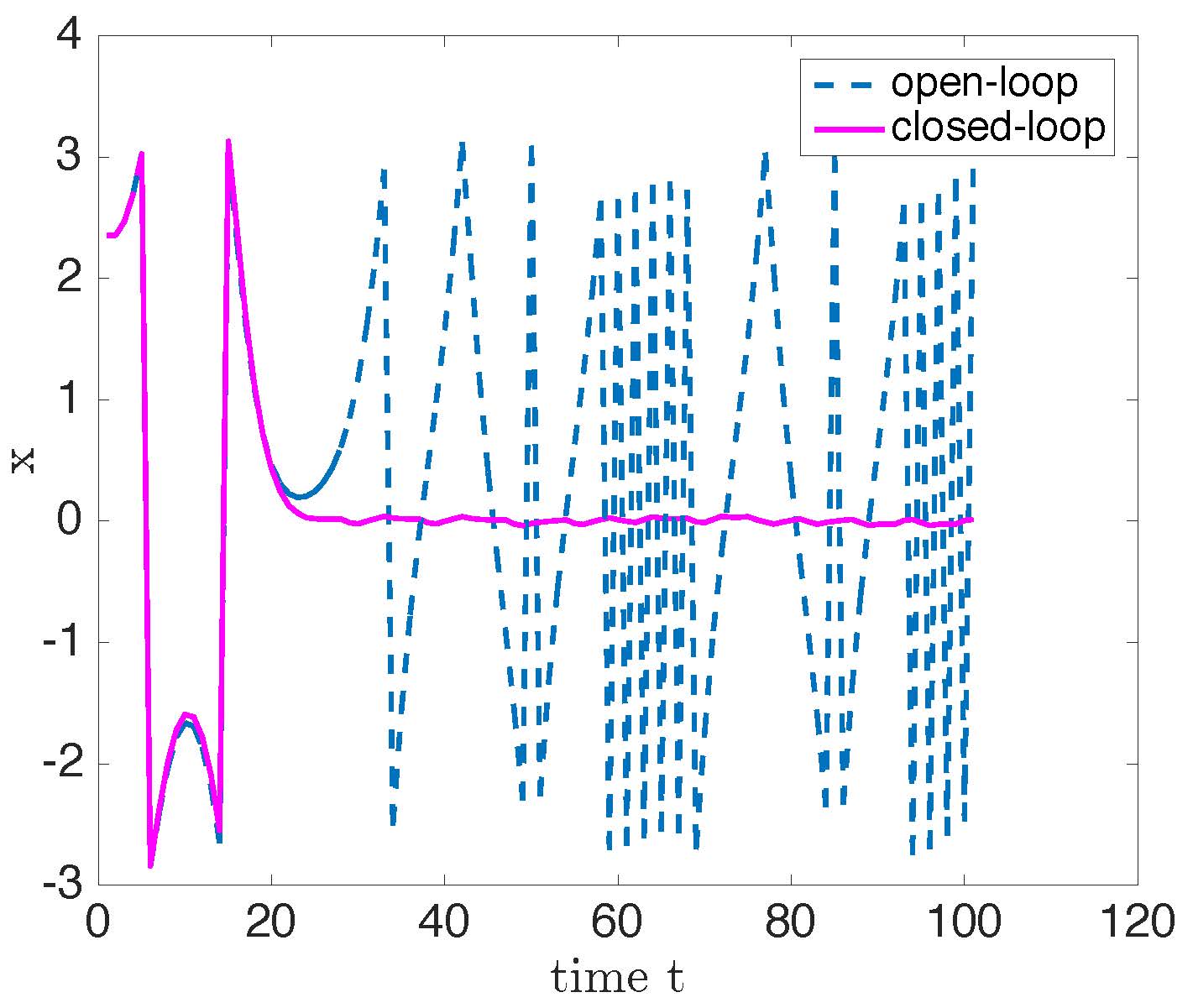}
\includegraphics[scale = 0.23]{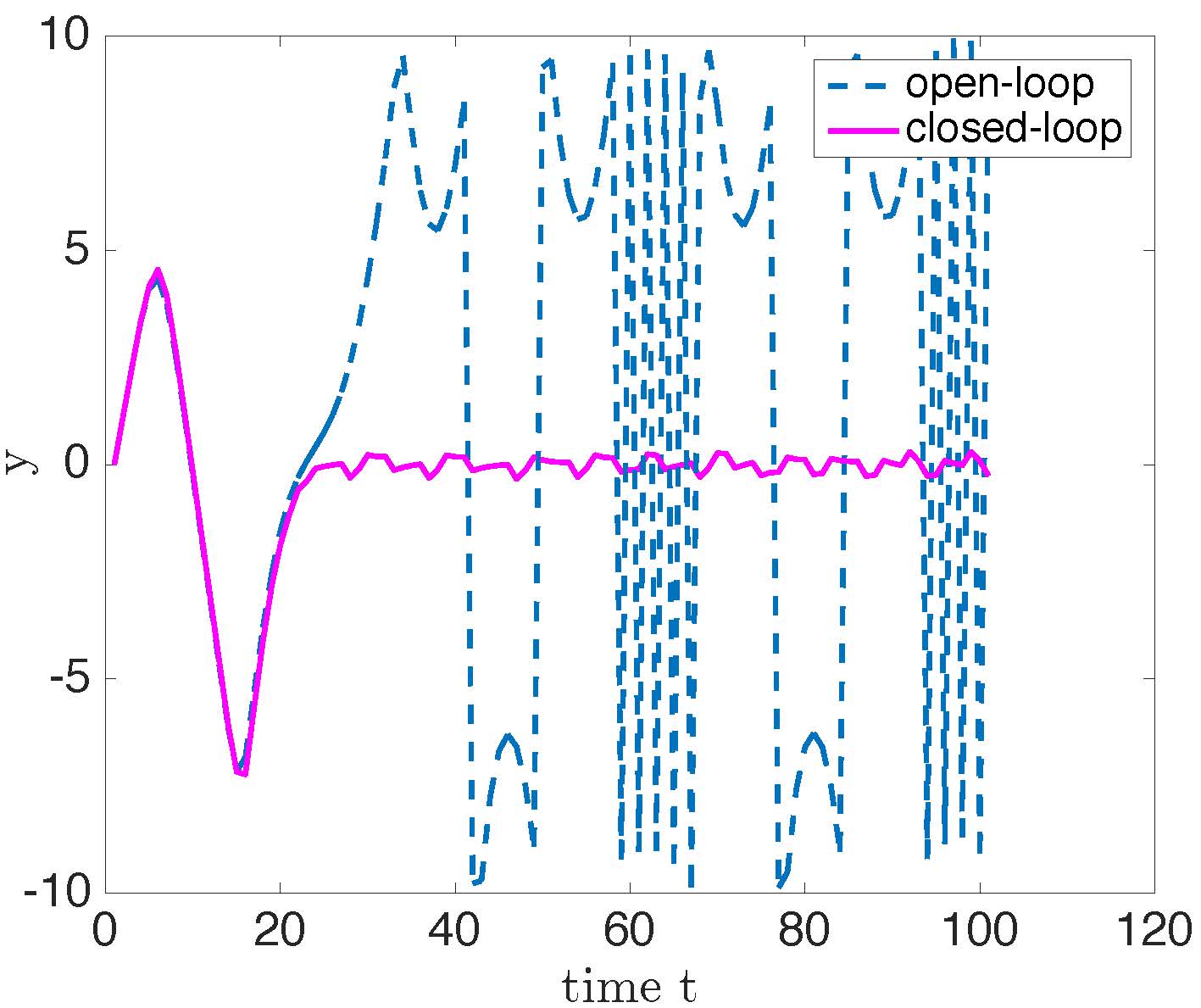}
\caption{Case 1: a) Sample $x$ trajectory; b) Sample $y$ trajectory}
\label{case1_fig3}
\end{center}
\end{figure}

\begin{figure}
\begin{center}
\includegraphics[scale = 0.10]{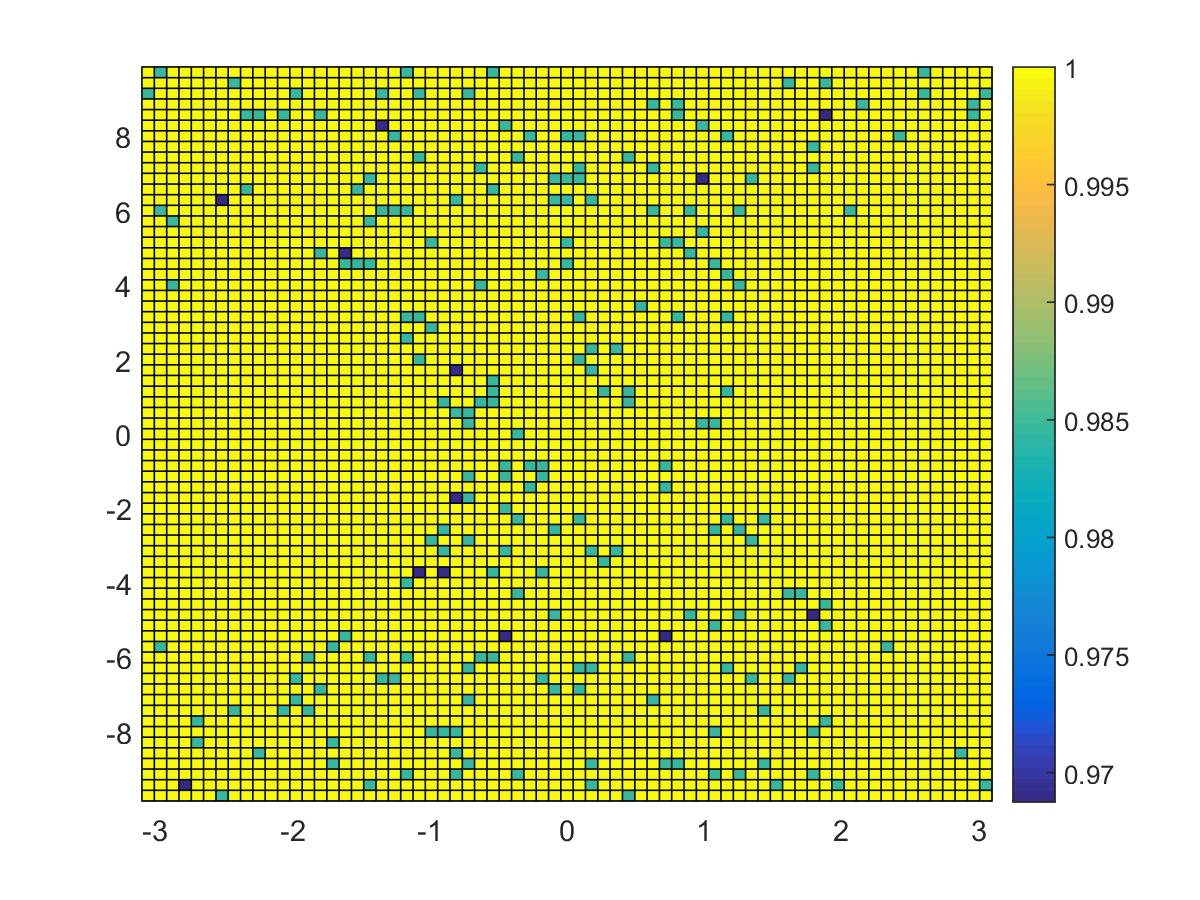}
\includegraphics[scale = 0.10]{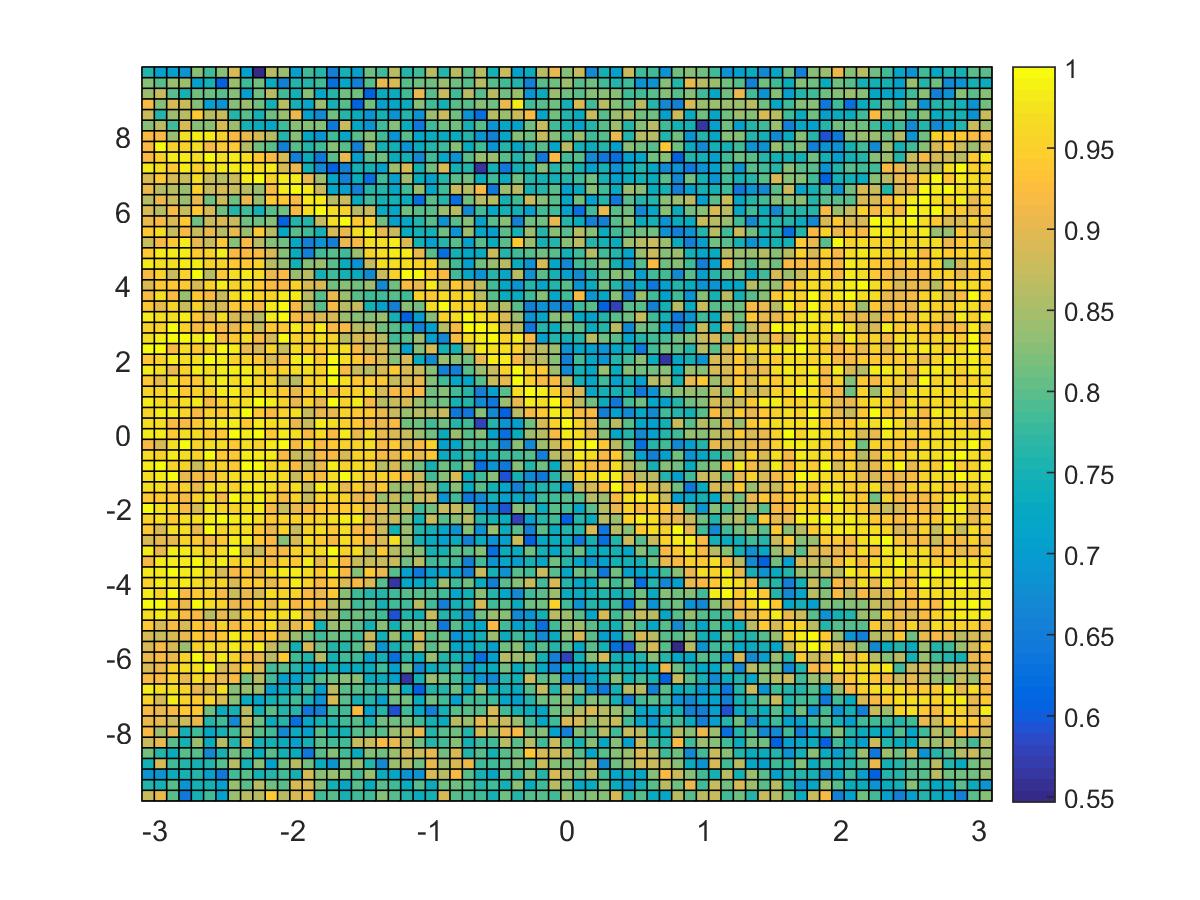}
\caption{Case 2: Plots for percentage of initial condition attracted to origin for erasure probability $1-p$ a) 0.15; b) 0.5.}
\label{case2_fig1}
\end{center}
\end{figure}